\documentclass[dvipsnames,letterpaper, 10 pt, journal]{ieeetran}

\IEEEoverridecommandlockouts	

\usepackage{cite}
\usepackage{amsmath,amssymb,amsfonts,amsthm}
\usepackage{algorithmic}
\usepackage{textcomp}
\usepackage{graphicx,color}
\usepackage{mathrsfs}
\usepackage{tikz}
\usepackage[vlined,ruled]{algorithm2e}
\usepackage{subfigure}
\usepackage{url}
\usepackage{color}
\usepackage{dsfont}
\usepackage{bbm}
\usepackage{booktabs}
\usepackage{array}
\usepackage{yfonts}
\usepackage[normalem]{ulem}
\usepackage{arydshln,leftidx,mathtools}
\usepackage{multicol}
\usepackage{amsmath}
\usepackage{stfloats}
\usepackage{comment}
\usepackage{subfigure}
\usepackage{manyfoot}
\usepackage{enumitem}

\definecolor{myBlue}{RGB}{49,130,189}

\newtheorem{theorem}{Theorem}[section]
\newtheorem{lemma}[theorem]{Lemma}
\newtheorem{definition}{Definition}

\newtheorem{remark}{Remark}

\newcommand{\map}[3]{#1: #2 \rightarrow #3}

\usepackage{tabularx}
\usepackage{multirow}
\usepackage{makecell}

\newcommand\aamsout{\bgroup\markoverwith{\textcolor{violet}{\rule[0.5ex]{2pt}{1pt}}}\ULon}

\newcommand{\T}{\mathsf{T}} 

\newcommand{\trace}{\mathrm{tr}}

\makeatletter
\g@addto@macro\normalsize{%
  \setlength{\abovedisplayskip}{4pt plus 1pt minus 1pt}%
  \setlength{\belowdisplayskip}{4pt plus 1pt minus 1pt}%
  \setlength{\abovedisplayshortskip}{2pt plus 1pt minus 1pt}%
  \setlength{\belowdisplayshortskip}{3pt plus 1pt minus 1pt}%
}
\makeatother

\DeclareSymbolFont{bbold}{U}{bbold}{m}{n}
\DeclareSymbolFontAlphabet{\mathbbold}{bbold}

\newcommand\oprocendsymbol{\hbox{$\square$}}
\newcommand\oprocend{\relax\ifmmode\else\unskip\hfill\fi\oprocendsymbol}

\renewcommand{\theenumi}{(\roman{enumi})}

\newcommand*{\QEDA}{\hfill\ensuremath{\blacksquare}}%

\DeclareNewFootnote{R}[roman]

\graphicspath{{figs/}}

\makeatletter
\let\NAT@parse\undefined
\makeatother
\usepackage[colorlinks,urlcolor=blue, linkcolor=blue, citecolor=blue]{hyperref}

\begin{document}

\title{\LARGE \bf Online Optimization with Unknown Time-Varying Parameters Using Noisy Gradient Measurements}

\author{Shivanshu~Tripathi,~Fabio~Pasqualetti,~and~Maziar~Raissi%
\thanks{S. Tripathi is with the Department of Electrical and Computer Engineering, and M. Raissi is with the Department of Mathematics, University of California, Riverside, Riverside, CA 92521, USA.
Emails: \href{mailto:strip008@ucr.edu}{\texttt{strip008@ucr.edu}} and
\href{mailto:maziarr@ucr.edu}{\texttt{maziarr@ucr.edu}}. F.~Pasqualetti is with the Department of Electrical and Computer Engineering, University of California, Irvine, Irvine,
CA 92697, USA. Email: \href{mailto:fabiopas@uci.edu}{\texttt{fabiopas@uci.edu}} }
}

\maketitle
\thispagestyle{empty}

\begin{abstract}
We study online optimization problems in which the cost function
depends on latent, time-varying parameters that are unmeasurable
and governed by unknown dynamics. Specifically, we consider a strongly
convex cost function whose linear term evolves according to unknown
linear stochastic dynamics about an unknown nominal operating point, while the algorithm has access only to finite noisy
gradient measurements. We propose a solution that uses control theoretic
tools to reconstruct the latent parameters from gradient observations
using a Gauss-Markov estimator, then identifies the parameter dynamics using
a bias-reduced instrumental variable estimator, and finally forecasts the
parameters to compute the minimizer. We provide a high-probability,
finite-horizon bound on the expected tracking error. 
We illustrate our algorithm on a series of numerical examples.
\end{abstract}

\section{Introduction}
\IEEEPARstart{T}{ime-varying} optimization problems arise in many real-time decision-making
systems, including robotics~\cite{PMW-MP-YH-AE-NM-ADP:24}, and
control~\cite{DPB:18}. 
In these
problems, the objective function depends on parameters that evolve over time
and are not directly observable, and the goal is to track the resulting moving
minimizer using only indirect information about the cost. This task becomes
especially challenging when the parameter trajectory, its governing dynamics,
and the available measurements are all uncertain.

Motivated by these challenges, we study online time-varying optimization with a
noisy gradient oracle. At each query point, the algorithm observes a noisy
gradient of the current objective and must use these measurements to infer the
underlying parameter evolution and track the corresponding optimizer. We propose
an online algorithm that first reconstructs the latent time-varying parameters
from short windows of noisy gradient measurements, then identifies their
dynamics and finally
predicts the parameter trajectory to compute the optimal solution. This leads to
a nonstandard identification problem: the parameters are never directly observed,
and the regressors used for dynamics estimation are themselves noisy.

\textbf{Related work.}
Methods for time-varying optimization split into
\emph{unstructured} approaches, which react to incoming gradients
without modeling their evolution
\cite{SSS:12,SCH-DS-JL-PZ:21} and therefore converge
only to a neighborhood of the optimal trajectory 
\cite{AS-ED-SP-GL-GBG:20,EDA-AS-SB-LM:120}, and
\emph{structured} approaches, which exploit the temporal variability of the model 
via prediction-correction schemes
\cite{AS-AM-AK-GL-AR:16,AS-ED:17,MF-SP-VMP-AR:17},
contracting continuous-time dynamics
\cite{AD-VC-AG-GR-FB:23}, or the internal model principle,
which requires the algorithm to embed a model of the parameter dynamics
for exact tracking \cite{GB-BVS:24,GB-BVS:25a}, with applications to
unconstrained quadratic \cite{NB-RC-SZ:24}, constrained
\cite{UC-NB-RC-SZ:25}, and feedback-optimization \cite{GB-BVS:25b}
settings. All of these works assume the parameter trajectory or its
dynamics are known. Our prior work \cite{ST-AAM-FP:24} relaxed this in a
\emph{noiseless} setting where exact recovery is possible. In this letter, we
consider both measurement and process noise, which turns identification
into a finite-sample estimation problem about an unknown operating point, 
requiring us to estimate the unknown parameters, and quantify the bound on tracking error. 
The resulting problem connects to finite-sample linear-system
identification \cite{TS-AR:19,SS-WH-XW:25},
but the latent parameter is unobserved and must
first be reconstructed from noisy gradients, so the regressors
share noise with the regression residual. This endogeneity biases
ordinary least squares and places the problem within errors
and instrumental-variable system identification
\cite{YZ-XZ-JL-NL:25,AG-KL:23}. The IV estimator we propose is
tailored to this dependence structure, and propagate the 
identification error through the forecast step to bound 
tracking error.

\noindent
\textbf{Contributions.} The main contributions of this letter are as follows.
First, we develop an online algorithm for tracking the minimizer of a
time-varying objective whose latent parameters evolve according
to unknown linear stochastic dynamics, about an unknown 
operating point. The proposed method uses only a
finite number of noisy gradient measurements. Second, we derive an
upper bound on the tracking error of the predicted minimizer. The
bound quantifies the error due to parameter estimation, dynamics
identification, and stochastic process noise. 
Third, we show how the estimated time-varying parameters can
be used to compute the corresponding minimizer using standard convex-optimization
algorithms, and we illustrate the effectiveness of our method
through a series of examples.

\noindent
\textbf{Notation.}
Let $\mathbb{R}$ denote set of real numbers and $I_n$ the $n\times n$
identity matrix. For a matrix $M$, $\|M\|_2$ denotes
Euclidean norm and $\|M\|_F$ denotes
Frobenius norm. For a square matrix $M$, $\rho(M)$, $\lambda_{\min}(M)$,
and $\mathrm{tr}(M)$ denote its spectral radius, minimum eigenvalue, and
trace, respectively. We write $M\succeq 0$ and $M\succ 0$ for positive semidefinite and
positive definite matrices, $M^\top$ denotes transpose. 
$\mathbb{E}[\cdot]$ denotes expectation, and $\mathcal{N}(0,\Sigma)$
denotes zero-mean Gaussian distribution with covariance matrix $\Sigma$, $\otimes$ for Kronecker product,
$\lfloor\cdot\rfloor$ and
$\lceil\cdot\rceil$ for floor and ceiling functions.

\section{Problem formulation}
We investigate the time-varying optimization problem:
\begin{align}\label{eq:prob}
\min_{x\in\mathbb{R}^n}~f(x,\theta(t))=g(x)^\T \theta(t),
\end{align}
where $\map{f}{\mathbb{R}^n \times \Theta}{\mathbb{R}}$ is the cost function
with unknown, unmeasurable, time-varying parameter
$\theta(t)\in\Theta\subset\mathbb{R}^p$, and $\map{g}{\mathbb{R}^n}{\mathbb{R}^p}$
is a known vector-valued function.
The unknown parameter $\theta(t)$ evolves under unknown linear, stochastic dynamics about an unknown nominal operating point $\bar \theta\in\mathbb{R}^p$:
\begin{align}\label{eq: dynamics}
\theta(t\!+\!1)=\bar \theta \!+\! A(\theta(t)\!-\!\bar\theta)\!+\!w_p(t),~ w_p(t)\sim\mathcal{N}(0,Q),
\end{align}
with unknown $A\in\mathbb{R}^{p\times p}$ and $Q\succeq 0$.
At each time \(t\), given \(x(t)\in\mathbb{R}^n\), the algorithm queries a gradient oracle and obtains
\begin{align}\label{eq: measurements}
    y(x(t),t)
    &= \nabla_x f(x(t),\theta(t)) + w_m(t) \notag \\
    &= \underbrace{ \left. \frac{\partial g^\top}{\partial x} \right|_{x} }_{C(x(t))} \theta(t) + w_m(t).
\end{align}
with $w_m(t)\sim\mathcal{N}(0,R), R\succ 0$, mutually independent of $w_p(t)$. Here
$C(x):=\partial g^\T / \partial x \in \mathbb{R}^{n\times p}$ is the gradient of $g^\T$, which is
a known function of $x$ since $g$ is known.
 
We assume that a finite number of measurements are available for $t\in [0,N)$. For all $t\geq N$, the
algorithm must produce $\hat x^*(t)$ using only the available measurement data. Therefore,
we aim to design an algorithm that produces the predicted minimizer $x^*(t)$ for $t\geq N$, and
to bound the expected prediction tracking error $\mathbb{E}\lVert \hat x^*(t)-x^*(t) \rVert_2$ as a
function of the number of available measurements $N$ and the prediction horizon
$h:=t-(N-1)\geq 1$.

\begin{definition}[\textbf{Convex region}]\label{def:Kmu}
For $\mu>0$, we define strong-convexity region
$  \mathcal{K}_\mu := \big\{\theta\in\mathbb{R}^p:\ \nabla_x^2 f(x,\theta)\succeq\mu I_n\ \ \forall x\in\mathbb{R}^n\big\}$.
Since $f(x,\theta)=g(x)^\top\theta$ is linear in $\theta$, we have
$\nabla_x^2 f(x,\theta)=\sum_{i=1}^{p}\theta_i\,\nabla_x^2 g_i(x)$, and
$\mathcal{K}_\mu$ is closed and convex.
\end{definition}
 
We impose the following assumptions throughout the letter:
 
\begin{itemize}
\item[(A1)] The map $g$ is twice continuously differentiable. The
    operating point lies in the strong-convexity region,
    $\bar\theta\in\mathcal{K}_\mu$.

    \item[(A2)] The matrix \(A\) is Schur stable, \(\rho(A)<1\),
    invertible, and the pair 
    \((A,Q^{1/2})\) is controllable.

    \item[(A3)] The sequence $\{x(t)\}_{t=0}^{N-1}$ lie in a set
$X\subseteq \mathbb{R}^n$, and there exist
$k_0\geq \lceil p/n \rceil$ and $\alpha_0>0$, such that
$
\sum_{j=0}^{k_0-1} C(x(t+j))^\top C(x(t+j)) \succeq \alpha_0 I_p,
\forall\, t\in\{0,\dots,N-k_0\},N>k_0$.
\end{itemize}
 
Assumption~(A1) ensures the existence of a unique minimizer trajectory 
\(x^*(t)\) on $\Theta$.    We define the margin and the operating
    region
$      \Delta:=\operatorname{dist}\big(\bar\theta,\partial\mathcal{K}_\mu\big)>0,
      \Theta:=\big\{\theta\in\mathbb{R}^p:\ \lVert\theta-\bar\theta\rVert_2\le\Delta\big\}$. By construction $\Theta\subseteq\mathcal{K}_\mu$, and consequently
$
      \nabla_x^2 f(x,\theta)\succeq\mu I_n,
\forall x\in\mathbb{R}^n,\ \forall\theta\in\Theta$. This minimizer is 
characterized by the first-order optimality condition 
\(\nabla_x f(x^*(\theta),\theta)=0\). By the implicit function 
theorem, the map \(\theta\mapsto x^*(\theta)\) is continuously differentiable
and \(L_*\)-Lipschitz on \(\Theta\), with 
$L_*=\frac{1}{\mu}\sup_{x\in\mathcal{X}^*}\|C(x)\|_2$, where 
\(\mathcal{X}^*=\{x^*(\theta):\theta\in\Theta\}\).
Assumption~(A2) guarantees that 
\(\theta(t+1)
\), admit a unique stationary covariance 
\(\Sigma_\theta\) satisfying 
\(\Sigma_\theta=A\Sigma_\theta A^\top+Q\). The controllability of 
\((A,Q^{1/2})\) further implies \(\Sigma_\theta\succ0\). {Assumption~(A3) is the 
standard persistence-of-excitation condition adapted to the known state-dependent 
matrix \(C(x)\). The condition can be verified with proper choice of $x(t)$ in the exogenous sequence $\{x(t)\}_{t=0}^{N-1}$ by designer. }
 
 \begin{remark}\label{rem: margin}
 For $\Delta>\sqrt{\operatorname{tr}\Sigma_\theta}$, union bound gives
$\mathbb{P}(\exists\,t\leq T:\theta(t)\notin\Theta)\leq(T+1)
\exp\big(-(\Delta-\sqrt{\operatorname{tr}\Sigma_\theta})^2/(2\lVert\Sigma_\theta\rVert_2)\big)$.
Hence under the margin-to-noise condition $\Delta\geq\sqrt{\operatorname{tr}\Sigma_\theta}
+\sqrt{2\lVert\Sigma_\theta\rVert_2\log(T/\delta)}$, assumption (A1) holds for all $t\leq T$
with probability at least $1-\delta$. Thus, $\theta(t)\in\Theta$ cannot hold surely.
\end{remark}
 
\section{Estimation of parameter dynamics and optimal solution}\label{section3}
In this section we describe the methodology to estimate the unknown parameter
$\theta(t)$ and its dynamics $A$ to solve the optimization problem \eqref{eq:prob}.
For notational convenience, let $y(t):=y(x(t),t)$, where $x(t)$ is the value at time $t$
at which the gradient $y(x,t)$ is evaluated. Let $Y$ be the data collected from the oracle
over the data collection period:
\begin{align}\label{eq:data}
Y = \begin{bmatrix} y(0)&\cdots&y(N-1) \end{bmatrix}\in\mathbb{R}^{n\times N}.
\end{align}
Our proposed approach has three stages. First, we use the Gauss-Markov theorem
to estimate $\theta(t)$ from $Y$ for $t< N$. Second, using the available 
estimates of $\theta(t)$ for $t< N$, we identify the
dynamics $A$ via instrumental variables. Third and finally, the identified dynamics are used to predict
$\theta(t)$ beyond the data-collection window, and the corresponding optimizer
is recovered from the optimality conditions.
 
\subsection{Estimation of $\theta(t)$ during data collection period}
For a fixed window length $k$ with $kn\geq p$, stack $k$ consecutive measurements as
$\bar y_t := [y(t)^\T, y(t+1)^\T, \cdots, y(t+k-1)^\T]^\T \in \mathbb{R}^{kn}$.
Substituting \eqref{eq: dynamics} into \eqref{eq: measurements} and propagating $\theta(t)$:
\begin{align}\label{eq: y}
\bar y_t = \bar C_t \theta (t) + b(t)+\bar w_{m_t},
\end{align}
where $\bar C_t:=[C(x(t))^\T,\cdots,C(x(t+k-1))^\T]^\T\in\mathbb{R}^{kn\times p}$,
$\bar w_{m_t}\sim \mathcal{N}(0,\bar R)$ with $\bar R:=I_k\otimes R$, and
\begin{align}\label{eq: bias}
b(t):=\begin{bmatrix} 0 \\ C(x(t+1)) \delta_1(t) \\ \vdots \\ C(x(t+k-1)) \delta_{k-1}(t) \end{bmatrix}
\end{align}
captures the drift with $\delta_j(t):=\theta(t+j)-\theta(t)=(A^j - I)\vartheta(t) +
\sum_{i=0}^{j-1} A^{j-1-i} w_p(t+i)$ with $\vartheta(t):=\theta(t)-\bar\theta$ and $\vartheta(0)\sim\mathcal{N}(0,\Sigma_\theta)$.

Under (A3), for $k\geq k_0$, $\bar C_t$ has full column rank $p$ with 
\begin{align}\label{eq: C}
\bar C_t^\T \bar R^{-1} \bar C_t \succeq \alpha_k I_p,\qquad \alpha_k:=\frac{\lfloor k/k_0\rfloor\alpha_0}{\lambda_{\text{max}}(R)}
\end{align}
holds uniformly over the window. This ensures a unique solution in the noise- and drift-free
limit and yields a uniformly bounded estimator covariance
$\Sigma_\eta(t)\preceq (1/\alpha_k)I_p$.
 
We design the estimator as if $b(t)$ were absent and account for the resulting error
separately. By Gauss-Markov theorem, the best linear unbiased estimator of
$\bar y_t=\bar C_t\theta(t)+\bar w_{m_t}$ is
\begin{align}\label{eq: tilde_theta}
\tilde \theta(t):= (\bar C_t^\T \bar R^{-1} \bar C_t )^{-1} \bar C_t^\T \bar R^{-1} \bar y_t.
\end{align}
 
\begin{lemma}[\textbf{Gauss-Markov}]\label{label1}
For $\bar y=\bar C\theta+\bar w$ with $\bar w\sim\mathcal{N}(0,\bar R)$, $\bar R\succ 0$,
and $\bar C$ of full column rank, the linear unbiased estimator $\hat \theta=K\bar y$ that
minimizes $\mathbb{E}[(\hat \theta-\theta)(\hat \theta-\theta)^\T]$ has
$K^*= (\bar C^\T \bar R^{-1} \bar C)^{-1} \bar C^\T \bar R^{-1}$ with covariance
$(\bar C^\T \bar R^{-1} \bar C)^{-1}$.
\end{lemma}
 
\begin{IEEEproof}
Unbiasedness, $\mathbb{E}[K\bar y]=K\bar C\theta=\theta$ for all $\theta$, is equivalent
to $K\bar C=I$. Setting $K^\star=(\bar C^\T \bar R^{-1}\bar C)^{-1}\bar C^\T \bar R^{-1}$
gives $K^\star\bar C=I_p$. For any other unbiased $K=K^\star+\Delta$, we have
$\Delta\bar C=0$, so the cross term
$K^\star \bar R\Delta^\T=(\bar C^\T \bar R^{-1}\bar C)^{-1}(\Delta\bar C)^\T=0$ vanishes.
Hence $K\bar R K^\T= K^\star \bar R(K^\star)^\T+\Delta \bar R\Delta^\T
\succeq K^\star \bar R(K^\star)^\T=(\bar C^\T \bar R^{-1}\bar C)^{-1}$,
with equality iff $\Delta=0$.
\end{IEEEproof}

Under (A3) with $k \ge k_0$, $\bar C_t$ has full rank.
Hence, $K^{*}_{t}\bar C_t = (\bar C_t^{\top}\bar R^{-1}\bar C_t)^{-1}
\bar C_t^{\top}\bar R^{-1}\bar C_t = I_p$. Applying $K^{*}_{t}$ to \eqref{eq: y}
gives 
\begin{align}\label{eq: b6}
  \tilde\theta(t) = K^{*}_{t}\bar y_t
  = \underbrace{K^{*}_{t}\bar C_t}_{=\,I_p}\theta(t) + \underbrace{K_t^* \bar w_{m_t}}_{\eta(t): \text{noise}} + \underbrace{K_t^* b(t)}_{\beta(t): \text{bias}}.
\end{align}
Since $w_m$ and $w_p$ are mutually independent, $\eta(t)$ and $\beta(t)$ are independent. The noise term $\eta(t)$ is zero-mean Gaussian with
\begin{align}\label{eq: cov}
\Sigma_\eta (t) \!=\! \mathbb{E}[\eta(t)\eta(t)^\T] \!=\! K_t^* \bar R (K_t^*)^\T\!=\! (\bar C_t^\T \bar R^{-1} \bar C_t)^{-1},
\end{align}
so by \eqref{eq: C}, $\lVert \Sigma_{\eta}(t) \rVert_2\leq 1/\alpha_k$. Jensen's inequality
then gives
\begin{align}\label{eq: b4}
\mathbb{E}\lVert \eta(t) \rVert_2 & \leq \sqrt{\trace(\Sigma_\eta(t))} \leq \sqrt{p/\alpha_k}.
\end{align}
For the bias term, $\lVert \beta(t) \rVert_2\leq\lVert K_t^*\rVert_2 \lVert b(t)\rVert_2$,
and \eqref{eq: cov} together with $\lambda_{\min}(\bar R)>0$ yields
$\lVert K_t^* \rVert_2 \leq \frac{1}{\sqrt{\alpha_k \lambda_{\text{min}}(\bar R)}}$.
From \eqref{eq: bias},
\begin{align}\label{eq: b(t)}
\lVert b(t) \rVert_2^2 \!=\! \sum_{j=1}^{k\!-\!1} \lVert C(x(t\!+\!j)) \delta_j(t) \rVert_2^2 \leq M^2\sum_{j=1}^{k\!-\!1} \lVert \delta_j(t) \rVert_2^2,
\end{align}
with $M:=\sup_{x\in\mathcal{X}} \lVert C(x) \rVert_2$. Splitting $\delta_j(t)=\delta_j^{\text{det}}(t)+\delta_j^{\text{stoch}}(t)$ with 
$\delta_j^{\text{det}}(t):=(A^j-I)\vartheta(t)$ and
$\delta_j^{\text{stoch}}(t):=\sum_{i=0}^{j-1} A^{j-1-i} w_p(t+i)$,
and using (A2) (under which there exist $C_A\geq 1$ and
$\bar\rho\in(\rho(A),1)$ with $\lVert A^m\rVert_2\leq C_A\bar\rho^m$ for
all $m\geq 0$, so that $\lVert A^j-I\rVert_2\leq c_A j\lVert A-I\rVert_2$
for a constant $c_A$ depending only on $A$), giving
$\lVert \delta_j^{\text{det}} (t) \rVert_2 \leq  {c_A}\, j \lVert A-I \rVert_2 \lVert \vartheta(t) \rVert_2,
\mathbb{E} \lVert \delta_j^{\text{stoch}}(t) \rVert_2^2 = \sum_{i=0}^{j-1} \trace (A^{j-1-i}Q(A^{j-1-i})^\T)\leq  {C_A^2}\, j\trace (Q)$.
Substituting into \eqref{eq: b(t)} and using
$\lVert\delta_j(t)\rVert_2^2\leq2\lVert\delta_j^{\text{det}}(t)\rVert_2^2
+2\lVert\delta_j^{\text{stoch}}(t)\rVert_2^2$,  and $\sum_{j=1}^{k-1}j^2\leq k^3/3$,
$\sum_{j=1}^{k-1}j\leq k^2/2$ gives
\begin{align}\label{eq: b2}
\!\!\mathbb{E}\lVert b(t) \rVert_2\! \leq\! M(c_1k^{3/2}\lVert A\!-\!I \rVert_2 \mathbb{E}\lVert \vartheta(t) \rVert_2\!+\!c_2 k \sqrt{\trace(Q)}),
\end{align}
where $c_1=\sqrt{2/3}c_A,c_2=C_A$ are constants and $B_\vartheta:=\sup_{t\geq0}(\mathbb{E}\lVert\vartheta(t)\rVert_2^2)^{1/2}
=\sqrt{\trace(\Sigma_\theta)}$. From \eqref{eq: b(t)} and \eqref{eq: b2}:
\begin{align}\label{eq: b5}
\mathbb{E}\lVert \beta(t) \rVert_2 \!\leq\! \frac{1}{\sqrt{\alpha_k}}
(C_1k^{3/2} \lVert A\!-\!I \rVert_2B_\vartheta+C_2k\sqrt{\trace(Q)})
\end{align}
with $C_i=Mc_i/\sqrt{\lambda_{\min}(\bar R)}$ for $i\in\{1,2\}$, and
$\lambda_{\min}(\bar R)=\lambda_{\min}(R)$ since $\bar R=I_k\otimes R$. Using \eqref{eq: C} and applying triangle inequality to \eqref{eq: b6} with \eqref{eq: b4} and \eqref{eq: b5},
\begin{align}\label{eq: trade}
\mathbb{E} \lVert \tilde \theta(t)-\theta(t) \rVert_2 \leq
\underbrace{{\sqrt{p/\alpha_k}}}_{\text{noise}}+
\underbrace{ C_1' k \lVert A-I \rVert_2 + C_2' \sqrt{k\trace {Q}} }_{\text{bias}}
\end{align}
where $C_1'=\frac{Mc_1B_\vartheta}{\sqrt{\alpha_0\,\lambda_{\min}(R)}}$ and $C_2'=\frac{Mc_2}{\sqrt{\alpha_0\,\lambda_{\min}(R)}}$.  The two terms in \eqref{eq: trade} pull in
opposite directions: larger $k$ shrinks the noise term as $1/\sqrt{k}$ while the drift
grows polynomially in $k$.

\subsection{Identification of operating point $\bar\theta$ and dynamics $A$}
For a fixed $k$, the data collection period yields $N-k+1$ estimates
$\{\tilde\theta(t)\}_{t=0}^{N-k}$ of the form \eqref{eq: tilde_theta}, which we stack as
$\tilde\Theta:=[\tilde\theta(0),\cdots,\tilde\theta(N-k)]\in\mathbb{R}^{p\times(N-k+1)}$.
Substituting $\theta(t)=\tilde\theta(t)-\eta(t)-\beta(t)$ from \eqref{eq: b6} into
\eqref{eq: dynamics} gives
\begin{align}\label{eq: b7}
\tilde\theta(t+1)\!=\!A\tilde\theta(t)\!+\! c\!+\!\xi(t),~ \forall~t\in\{0,\cdots,N-k-1\}
\end{align}
with intercept $c:=(I-A)\bar\theta$ and regression noise
\begin{align*}
\xi(t):= \underbrace{w_p(t)}_{\text{process noise}} + \underbrace{\eta(t+1)-A\eta(t)}_{\text{measurement noise residual}} + \underbrace{\beta(t+1)-A\beta(t)}_{\text{bias residual}}.
\end{align*}
 We identify $A$ and $c$ jointly to recover $\bar\theta=(I-A)^{-1}c$, which is well posed as $\rho(A)<1$ implies that $I-A$ is invertible.

\begin{remark}[OLS bias]\label{rem: OLS}
We do not use Ordinary least squares (OLS) as it requires
$\mathbb{E}[\tilde\theta(t)\xi(t)^\T]=0$, but both $\tilde\theta(t)$ and $\xi(t)$ contain
$\eta(t)$, producing the nonzero 
$\mathbb{E}[\eta(t)(-A\eta(t))^\T]=-\Sigma_\eta(t)A^\T$ inside
$\mathbb{E}[\tilde\theta(t)\xi(t)^\T]$ by \eqref{eq: cov}. 
\end{remark}
 
 We identify $G:=[A~c]\in\mathbb{R}^{p\times(p+1)}$ using instrumental variables (IV) estimator with $\Phi:=[\tilde\theta(t)~~1]^\T$ and $z(t):=[\tilde\theta(t-k)~~1]^\T\in\mathbb{R}^{p+1}$, so that $\tilde\theta(t+1)=G\Phi(t)+\xi(t)$. 
Multiplying \eqref{eq: b7} by $z(t)^\T$ and taking the expectation gives 
\begin{align}
\begin{aligned}
M_1&=GM_0+\Xi,\\
\Xi:&=\mathbb{E}[\xi(t)z(t)^\T]=\mathbb{E}[(\beta(t+1)-A\beta(t))z(t)^\T],
\end{aligned}
\end{align}
where $M_0:=\mathbb{E}[\Phi(t)z(t)^\T],M_1:=\mathbb{E}[\tilde\theta(t+1)z(t)^\T]$. Hence $M_1M_0^{-1}=G+\Xi M_0^{-1}$ whenever $M_0$ is invertible. 

Since $\bar C_t$ is deterministic under (A3) and $\mathbb{E}[b(t)]=0$, we have
$\mathbb{E}[\beta(t)]=0$ and $\mathbb{E}[\tilde\theta(t)]=\bar\theta$. Thus,
\begin{align*}
  &M_0\!=\!\begin{bmatrix}
      A^k\Sigma_\theta\!+\!\bar\theta\bar\theta^\T\!+\!\Delta_0 & \bar\theta\\[2pt]
      \bar\theta^\T & 1\end{bmatrix}\!,\\
  &\Delta_0\!:=\! \mathbb{E}[\vartheta(t)\beta(t\!-\!k)^\T]\!+\!\mathbb{E}[\beta(t)\vartheta(t\!-\!k)^\T]+\mathbb{E}[\beta(t)\beta(t\!-\!k)^\T],
\end{align*}
where $\Delta_0$ collects the bias contributions of \eqref{eq: b6} and vanishes for an
unbiased estimator. The procedure relies on $M_0$ being invertible. We characterize this in next result.

\begin{lemma}[\textbf{Condition for Invertibility of $M_0$}]\label{lem: invertibility}
Let (A1)--(A3) hold and $\nu_k:=\sigma_{\min}(A^k\Sigma_\theta)>0$. If
$\lVert\Delta_0\rVert_2\leq\nu_k/2$, then $M_0$ is invertible and
$\lVert M_0^{-1}\rVert_2\leq1+2(1+\lVert\bar\theta\rVert_2^{2})/\nu_k$. 
\end{lemma}

\begin{IEEEproof}
As an entry of $M_0$ is $1$, $M_0$ is invertible iff its Schur
$S=A^k\Sigma_\theta+\Delta_0$ is. Weyl's inequality gives
$\sigma_{\min}(S)\geq\nu_k-\lVert\Delta_0\rVert_2\geq\nu_k/2>0$, so
$\lVert S^{-1}\rVert_2\leq2/\nu_k$, giving
$M_0^{-1}=\bigl[\begin{smallmatrix}I\\-\bar\theta^\T\end{smallmatrix}\bigr]S^{-1}
\bigl[\begin{smallmatrix}I&-\bar\theta\end{smallmatrix}\bigr]
+\bigl[\begin{smallmatrix}0&0\\0&1\end{smallmatrix}\bigr]$ and
$\bigl\lVert\bigl[\begin{smallmatrix}I&-\bar\theta\end{smallmatrix}\bigr]\bigr\rVert_2^{2}
=1+\lVert\bar\theta\rVert_2^{2}$.
\end{IEEEproof}

{  Also, for $\hat M_0$, if $\lVert\hat M_0-M_0\rVert_2\leq\varepsilon_0
:=\min\{1,\nu_k/(8(1+\lVert\bar\theta\rVert_2))\}$, then $\lVert\hat M_0^{-1}\rVert_2\leq1+12(1+\lVert\bar\theta\rVert_2^{2})/\nu_k$.}

Replacing population by sample averages with $\hat M_0:=\frac{1}{N-2k}\sum_{t=k}^{N-k-1}\Phi(t)z(t)^\T, \hat M_1=\frac{1}{N-2k}\sum_{t=k}^{N-k-1}\tilde\theta(t+1)z(t)^\T$ yields the bias-reduced IV estimator
\begin{align}\label{eq: IV}
\begin{bmatrix} \hat A_N~~\hat c_N \end{bmatrix}=\hat G_N=\hat M_1\hat M_0^{-1},~\hat{\bar\theta}_N:=(I-\hat A_N)^{-1}\hat c_N.
\end{align}
\begin{lemma}\label{lem: IV}
Under Assumptions~(A1)-(A3) {and $\lVert\Delta_0\rVert_2\leq\nu_k/2$}, the IV estimator \eqref{eq: IV} for every $\delta\in(0,1)$ satisfies
\begin{align}\label{eq: Arate}
\lVert \hat G_N \!-\! G \rVert_2\leq \underbrace{C_A\sqrt{\frac{\log(N/\delta)}{N}}}_{\text{finite-sample error}}+\underbrace{ {C_B\bigl(k\lVert A\!-\!I\rVert_2\!+\!\sqrt{k\,\trace(Q)}\bigr)}}_{\text{bias residual}},
\end{align}
with probability at least $1-\delta$, for constants $C_A,C_B$. 
\end{lemma}
 
\begin{IEEEproof}
Split the error as $\hat G_N-G=(\hat G_N-M_1M_0^{-1})+\mathcal{E}$, with
$\mathcal{E}:=M_1M_0^{-1}-G$, and bound each term as:
\begin{align*}
\hat G_N&-M_1M_0^{-1}=(\hat M_1-M_1)M_0^{-1}\\
&\quad -M_1M_0^{-1}(\hat M_0-M_0)M_0^{-1}+\mathcal{O}(\lVert \hat M-M\rVert_2^2).
\end{align*}
Under (A2)--(A3) the process $\{\vartheta(t),\eta(t)\}$ is jointly Gaussian and
geometrically mixing, so each entry of $\hat M_i-M_i$ is a centered average of
products of two such variables. Bernstein-type concentration for sub-exponential
mixing sequences gives $\lVert\hat M_i-M_i\rVert_2\leq c\sqrt{\log(N/\delta)/N}$
for $i\in\{0,1\}$ with probability at least $1-\delta$. Combined with the bound on
$\lVert M_0^{-1}\rVert_2$ from Lemma~\ref{lem: invertibility} and the uniform
boundedness of $\lVert M_1\rVert_2$, this yields the first term of
\eqref{eq: Arate}.
Since $w_p(t)$, $\eta(t)$, and $\eta(t+1)$ are each
uncorrelated with $z(t)$, $M_1=GM_0+\Xi$ gives
$\mathcal{E}=\Xi M_0^{-1}$. By Jensen's and Cauchy--Schwarz inequalities,
$\lVert\Xi\rVert_2\leq(1+\lVert A\rVert_2)
\max_{s\in\{t,t+1\}}\bigl(\mathbb{E}\lVert\beta(s)\rVert_2^2\bigr)^{1/2}
\bigl(\mathbb{E}\lVert z(t)\rVert_2^2\bigr)^{1/2}$,
so \eqref{eq: b5}, the uniform bound
$B_z:=\sup_t(\mathbb{E}\lVert z(t)\rVert_2^2)^{1/2}<\infty$, and
Lemma~\ref{lem: invertibility} give
$\lVert\mathcal{E}\rVert_2\leq C_B\bigl(k\lVert A-I\rVert_2
+\sqrt{k\,\trace(Q)}\bigr)$. 
The triangle inequality completes the proof.
\end{IEEEproof}

 The bound \eqref{eq: Arate} holds for $\hat A_N$ and $\hat c_N$ (submatrices of $\hat G_N$). 
 
\subsection{Computing optimal solution}

For $t\geq N$, no measurements are available, so we propagate the most recent 
estimate $\tilde\theta(N-k)$ forward through the identified dynamics about the identified operating point:
\begin{align}\label{eq: forward}
\hat\theta(t):=\hat{\bar\theta}_N+\hat A_N^{\,h+k-1}\,(\tilde\theta(N-k)-\hat{\bar\theta}_N),\quad h=t-(N-1).
\end{align}
The predicted minimizer $\hat x^*(t)$ is the unique solution obtained using any standard convex solver as: 
\begin{align}\label{eq: final}
C(\hat x^*(t))\,\hat\theta(t)=0,
\end{align}
which, by (A1), is the gradient-zero condition of a strongly convex function and admits
a unique solution. We now bound the expected tracking error against $x^*(t)$.

{\begin{remark}[\textbf{Computational complexity}]\label{rem: cost}
The identification stage is executed once and scales polynomially with $N$ and with
the dimensions of $x$ and $\theta$. For $t\geq N$ the forecast \eqref{eq: forward}
requires only a matrix power, so the per-step cost is that of the static problem
\eqref{eq: final}, which recurs at every instant and limits the method in high
decision dimensions.
\end{remark}}

\begin{theorem}\label{thm: main}
Under (A1)--(A3), { $\lVert\Delta_0\rVert_2\leq\nu_k/2$,} and
$N\geq 2k+p+1$, the predicted 
$\hat x^*(t)$ in \eqref{eq: final} satisfies, for $t\geq N$ and
$\delta\in(0,1)$, which holds with probability at least $1-\delta$,
\begin{align}\label{eq: c1}
&\mathbb{E}\bigl[\lVert \hat x^*(t)-x^*(t)\rVert_2\,
 \bigr]
 \leq L_*\Biggl[\Bigl(C_4'\,H\bar\rho^{\,H-1}
 +C_3'\lVert I-A^H\rVert_2\Bigr)\notag\\[-0.9pt]
&\times\Biggl(\sqrt{\frac{\log(N/\delta)}{N}}
 +k\lVert A-I\rVert_2+\sqrt{k\,\trace(Q)}\Biggr)\\[-0.9pt]
&+\lVert A^H\rVert_2\Bigl(\sqrt{p/\alpha_k}
 +C_1'\,k\lVert A-I\rVert_2+C_2'\sqrt{k\,\trace(Q)}\Bigr)+\sigma_H \Biggr] \notag
\end{align}
where $L_*$ is Lipschitz constant, $H=t-(N-k)$,
$\sigma_H=\sqrt{\trace\bigl(\textstyle\sum_{j=0}^{H-1}A^jQ(A^j)^\T\bigr)}$ and
$C_1',C_2',C_3',C_4'$ are constants. { The expectation is taken over
$\vartheta(0)$, $\{w_p(s)\}_{s\geq0}$ and $\{w_m(s)\}_{s<N}$, and
$\mathcal{E}_\delta$ is the intersection of the event of Lemma~\ref{lem: IV}, on
which $\hat M_0$ is invertible, so that
$\mathbb{P}(\mathcal{E}_\delta)\geq1-\delta$.}
\end{theorem}

\begin{IEEEproof}
By (A1), $\theta\mapsto x^*$ is $L_*$-Lipschitz, so
$\lVert\hat x^*(t)-x^*(t)\rVert_2 \leq L_* \lVert\hat \theta(t)-\theta(t)\rVert_2$.
The true parameter evolves as
\begin{align}\label{eq: A1}
\theta(t)\!=\!\bar\theta\!+\! A^H(\theta(N\!-\!k)\!-\!\bar\theta)\!+\!\!\sum_{j=0}^{H-1}\! A^j w_p(t\!-\!1\!-\!j).
\end{align}
Subtracting \eqref{eq: A1} from \eqref{eq: forward} and applying triangle inequality,
\begin{align}\label{eq: A2}
&\lVert \hat \theta(t)\!-\!\theta(t) \rVert_2\!\leq\!  \underbrace{\lVert (\hat A_N^H \!-\! A^H)\tilde\vartheta(N\!-\!k) \rVert_2}_{\text{(I) model error}} \!+\! \underbrace{\lVert (I\!\!-\!\!A^H)(\hat{\bar\theta}_N\!\!-\!\!\bar\theta) \rVert_2}_{\text{(III) operating point error}}  \notag \\
& \!+\!\underbrace{\lVert A^H(\tilde \theta(N\!-\!k)-\theta(N\!-\!k)) \rVert_2}_{\text{(II) anchor error}} \! + \! \underbrace{\lVert \textstyle\sum_{j=0}^{H-1} A^j w_p(t\!-\!1\!-\!j) \rVert_2}_{\text{(IV) future process noise}}.
\end{align}

\emph{(I) Model error.} We have
$\hat A_N^H-A^H=\sum_{i=0}^{H-1} \hat A_N^i(\hat A_N-A) A^{H-1-i}$,
together with (A2) ($\rho(A)<1$) and $\rho(\hat A_N)<1$ with high probability for $N$ large enough that
Lemma~\ref{lem: IV} concentrates, there exist $\bar\rho\in(\rho(A),1)$ and $C_1>0$ with
$\lVert \hat A_N^i\rVert_2,\lVert A^i\rVert_2\leq C_1\bar\rho^{\,i}$. Submultiplicativity then gives
\begin{align}\label{eq: A4}
\lVert \hat A_N^H-A^H\rVert_2\leq C_1 H \bar\rho^{\,H-1}\lVert \hat A_N-A\rVert_2.
\end{align}
Defining $\tilde\vartheta(N-k):=\tilde\theta(N-k)-\hat{\bar\theta}_N$. Under (A2), $\vartheta(t)$ is stationary with covariance $\Sigma_\theta$, so
$\mathbb{E}\lVert\vartheta(N-k)\rVert_2 \leq \sqrt{\trace(\Sigma_\theta)}$ by Jensen's inequality. Combining
\eqref{eq: b6}, \eqref{eq: b4}, and \eqref{eq: b5} gives
\begin{align}\label{eq: theta_bd}
&\mathbb{E}\lVert \tilde\vartheta(N-k) \rVert_2 \leq \sqrt{\trace(\Sigma_\theta)}+\sqrt{p/\alpha_k} \notag\\
&\!+\! C_1' k \lVert A\!-\!I\rVert_2 \!+\!C_2'\sqrt{k\,\trace(Q)}+\mathbb{E}\lVert \hat{\bar\theta}_N\!-\!\bar\theta \rVert_2\!=:\!B_\theta.
\end{align}
 {Conditioning on the identification event of Lemma~\ref{lem: IV},
the bound on $\lVert\hat A_N-A\rVert_2$ is deterministic; combining
\eqref{eq: A4} with \eqref{eq: Arate} and \eqref{eq: theta_bd} via the
Cauchy-Schwarz inequality (applied to the remaining expectation over
$\bar w_{m_{N-k}}$) gives}
\begin{align}\label{eq:A5}
&\mathbb{E}\!\left[
\left\|(\hat A_N^H - A^H)\tilde\vartheta(N-k)\right\|_2
\right]
\leq
 \notag\\
&
C_4' H \bar\rho^{\,H-1}
\Bigl(
\sqrt{\frac{\log(N/\delta)}{N}} + k\|A-I\|_2 + \sqrt{k\,\operatorname{tr}(Q)}
\Bigr)
\end{align}
on the same identification event, where $C_4'$ collects $C_1$, $B_\theta$, and the constants
from Lemma~\ref{lem: IV}. 

\emph{(II) Anchor error.} Submultiplicativity and \eqref{eq: trade} give
\begin{align}\label{eq: A6}
&\mathbb{E}\lVert A^H(\tilde\theta(N-k)-\theta(N-k))\rVert_2 \notag\\
&\leq \lVert A^H\rVert_2\bigl(\sqrt{p/\alpha_k}+C_1' k\lVert A-I\rVert_2+C_2'\sqrt{k\,\trace(Q)}\bigr).
\end{align}

\emph{(III) Operating-point error.} The bound on $\lVert \hat{\bar\theta}_N-\bar\theta \rVert_2$ gives
\begin{align}\label{eq:op-pt}
\lVert (I-&A^H)(\hat{\bar\theta}_N-\bar\theta) \rVert_2 \leq C_3'\lVert I-A^H \rVert_2 \times \nonumber \\
&(\sqrt{\log(N/\delta)/N}+k\lVert A-I \rVert_2+\sqrt{k\trace(Q)})
\end{align}
with $C_3'=2\lVert (I-A)^{-1} \rVert_2(1+\lVert \bar\theta \rVert_2)\max\{C_A,C_B\}$.

\emph{(IV) Future process noise.} The sum is zero-mean Gaussian with covariance
$\sum_{j=0}^{H-1} A^j Q(A^j)^\T$, so by Jensen's inequality,
\begin{align}\label{eq: A7}
\mathbb{E}\biggl\lVert \textstyle\sum_{j=0}^{H-1} A^j w_p(t\!-\!1\!-\!j)\biggr\rVert_2\! \leq\! \sqrt{\textstyle\trace\sum_{j=0}^{H-1} A^j Q(A^j)^\T}.
\end{align} 
 
Substituting~\eqref{eq:A5},
\eqref{eq: A6}, \eqref{eq:op-pt}, and~\eqref{eq: A7} into~\eqref{eq: A2}, taking
expectations { over the future process noise and the anchor
measurement-noise window}, and applying the $L_*$-Lipschitz property gives the
result~\eqref{eq: c1} { on the identification event of
Lemma~\ref{lem: IV}, which holds} with probability at least $1-\delta$.
\end{IEEEproof}

\begin{remark}
The four terms in \eqref{eq: c1} behave differently. The model-error term
$C_4' H \bar\rho^{\,H-1}\sqrt{\log(N/\delta)/N}$ vanishes as $N\to\infty$ at the rate
$\sqrt{\log N/N}$ inherited from Lemma~\ref{lem: IV} {, plus a non-vanishing identification-bias contribution of order $k\lVert A-I\rVert_2+\sqrt{k\,\trace(Q)}$ that scales with the window length $k$}. The anchor-error term carries the
prefactor $\lVert A^H \rVert_2$, which is uniformly bounded in $H$ and decays geometrically
under stable dynamics.  The operating-point term instead has $\lVert I-A^H \rVert_2\leq 1+\lVert A^H \rVert_2$, which does not decay in $H$: an error in $\hat{\bar\theta}_N$ offsets the forecast and as $H\to\infty$ the forecast collapses to $\hat{\bar\theta}_N$. The prediction floor $\sqrt{\trace\sum_{j=0}^{H-1} A^j Q(A^j)^\T}$
grows with $H$ and saturates at $\sqrt{\trace(\Sigma_\theta)}$, independently of $N$.
Consequently, for short horizons and large $N$, the anchor term dominates and tighter
per-step recovery, via the choice of $k$, is most effective; for long horizons, the
prediction floor dominates, and only a fundamentally smaller process noise $Q$ or more data can reduce
the error further. Throughout, $k$ is treated as a fixed
design parameter (chosen to satisfy $kn\geq p$).
\end{remark}

{\begin{remark}[\textbf{Computing the minimizer}]\label{rem: solve}
Once the parameter dynamics are identified, the minimizer of \eqref{eq:prob} is
recovered by solving \eqref{eq: final} with an algorithm suited to 
the cost. For a quadratic cost a closed-form expression is available, as in
Section~\ref{eq: RTTT}. For a general cost, \eqref{eq: final}  is solved by gradient descent
as described in Section~\ref{eq: NCF}.
\end{remark}}

\section{Numerical Examples and Comparisons}
We now present numerical studies on two different tasks: 
\subsection{Real-Time Trajectory Tracking (Quadratic Cost Function)}\label{eq: RTTT}

\begin{figure*}[!t]
  \centering
  \includegraphics[width=2\columnwidth,trim={0cm 0cm 0cm 0cm},clip]{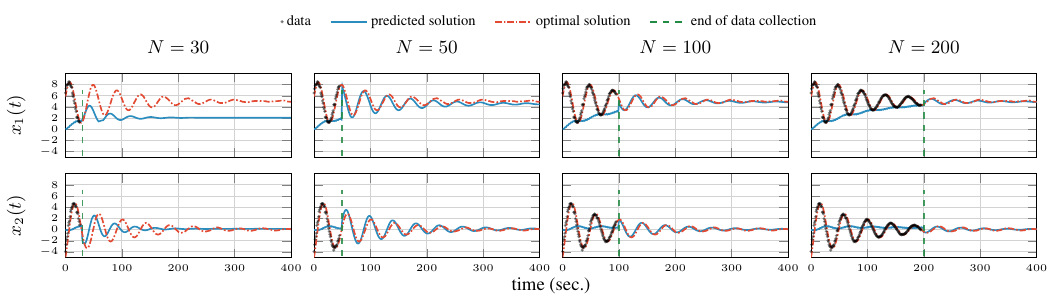}
  \caption{This figure shows the predicted (solid blue line) and the optimal
solution (dashed red line) considering a quadratic cost function for the setting
 in Section~\ref{eq: RTTT} for $N\in\{30,50,100,200\}$ over
  $t\in[0,400]$. Each row corresponds to the
first and the second component of the predicted and the optimal solution,
respectively. Gray dots indicate the measurements collected during the data-collection interval  {$[0,N)$}, and the vertical dashed green line marks the end of this interval. For $t\in[0,N)$, the predicted solution is computed using static gradient descent, whereas for $t\geq N$ it is computed using~\eqref{eq: final}. The results show that static gradient descent exhibits a lag in tracking the optimal solution \cite{ST-AAM-FP:24}, while the proposed method reduces the tracking error as $N$ increases.  
  }
  \label{fig: rmse_vs_N}
\end{figure*}

\begin{figure}[!t]
  \centering
  \includegraphics[width=\columnwidth,trim={0cm 0cm 0cm 0cm},clip]{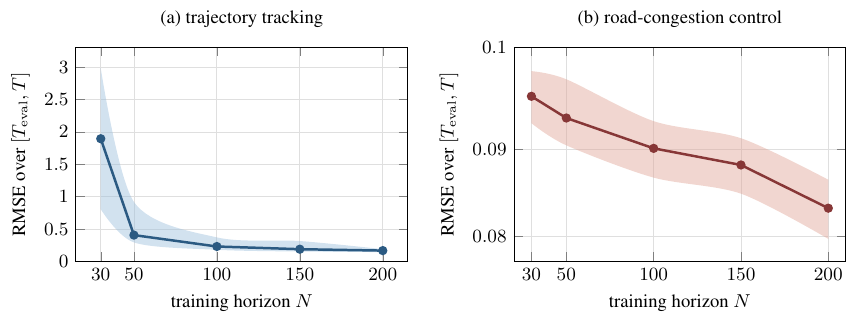}
  \caption{ The figure shows RMSE of the predicted optimizer over the evaluation window $[T_{\mathrm{eval}},T]$ for training horizons $N\in\{30,50,100,150,200\}$, averaged over $M=30$ Monte Carlo trials. (a) Trajectory tracking (Section~\ref{eq: RTTT}), evaluated on $[200,400]$; (b) road-congestion control (Section~\ref{eq: NCF}), evaluated on $[100,200]$. In both cases, more noisy gradient measurements yield a lower future tracking error, indicating improved estimation of parameter dynamics.
 }
  \label{fig: tracking_trajectories}
\end{figure}

We apply the proposed framework to a real-time trajectory tracking task in which a ground
robot tracks an uncrewed aerial vehicle (UAV) with unknown trajectory and unknown
trajectory dynamics, using only noisy gradient measurements. At each time $t$, the robot
solves
$\min_{x\in\mathbb{R}^2}\ f(x,t) = \bigl(x - b(t)\bigr)^{\!\top} H(t) \bigl(x - b(t)\bigr),$
where $x \in \mathbb{R}^2$ is the position of the robot, $b(t)\in\mathbb{R}^2$ is
the unknown UAV position, and
$H(t)=\bigl[\begin{smallmatrix} h_{11}(t) & h_{12}(t)\\ h_{12}(t) & h_{22}(t)\end{smallmatrix}\bigr]\succ 0$
is an unknown symmetric weighting matrix encoding task priorities (e.g., velocity or position). 
Defining $\tilde b(t):=H(t)b(t)$ and  {$\kappa(t):=b(t)^{\!\top} H(t)b(t)$} recasts
the problem into the parametric form  {$f(x,t) = g(x)^{\!\top}\theta(t) + \kappa(t)$}
with
$g(x) = \begin{bmatrix} -2x_1 & -2x_2 & x_1^2 & 2x_1 x_2 & x_2^2 \end{bmatrix}^\T,
\theta(t) = \begin{bmatrix} \tilde b_1(t) & \tilde b_2(t) & h_{11}(t) & h_{12}(t) & h_{22}(t) \end{bmatrix}^\T$.
The bias  {$\kappa(t)$} does not depend on $x$, and the minimizer is recovered in closed form as
$x^\star(t)=b(t)=H(t)^{-1}\tilde b(t)$. The parameter $\theta(t)\in\mathbb{R}^5$ evolves
under unknown linear dynamics \eqref{eq: dynamics} with
$w_p(t)\sim\mathcal{N}(0,\sigma_p^2 I_5)$. At each $t\in[0,N)$, the robot receives a
noisy gradients $y(t) = C(x(t))\theta(t) + w_m(t)$, 
$C(x) = \left[\begin{smallmatrix} -2 & 0 & 2x_1 & 2x_2 & 0 \\ 0 & -2 & 0 & 2x_1 & 2x_2 \end{smallmatrix}\right]$, and $A=\left[\begin{smallmatrix} 1.04 & -0.16 & 0    &  0    & 0    \\
 0.16 &  0.92 & 0    &  0    & 0    \\
 0    &  0    & 0.95 &  0.01 & 0.01 \\
 0    &  0    & -0.01 & 0.94 & 0.01 \\
 0    &  0    & -0.01 & 0.01 & 0.93\end{smallmatrix}\right]$.

Since $\nabla_x^2f(x,t)=2H(t)$, the strong-convexity region is
$\mathcal{K}_\mu=\{\theta:\lambda_{\min}(H(t))\geq\mu/2\}$. The operating point is $\bar \theta=[\bar{\tilde b},\bar h_{11},\bar h_{12},\bar h_{22}]^\T$ with $\bar H=5\left[\begin{smallmatrix}1&0.2\\0.2&1\end{smallmatrix}\right]$ and $\bar b=[5,0]^\T$, and
$\lambda_{\min}(\bar H)=4$; with $\mu=2$ this gives a margin
$\Delta\approx2.45$. The realized stationary spread is
$\sqrt{\trace(\Sigma_\theta)}=0.36$, so the margin-to-noise condition of
Remark~\ref{rem: margin} holds over $[0,T]$ at $\delta=0.05$.

We set the horizon $T=400$, evaluation window
$[T_{\mathrm{eval}},T]=[200,400]$, training horizon $N\in\{30,50,100,200\}$, window
 {$k=3$}, measurement noise $\sigma_m=0.80$, and process noise $\sigma_p=0.035$.  Each
configuration runs for $30$  {Monte Carlo} trials. 
Performance is measured as:
$\mathrm{RMSE}
=\biggl[\frac{1}{T-T_{\text{eval}}+1}
\sum_{t=T_{\mathrm{eval}}}^{T}
\|\hat x^\star(t)-x^\star(t)\|_2^2
\biggr]^{1/2}$.
For $t<N$, the robot uses the static gradient
descent, $x(t+1)=x(t)-\eta\, y(t)$ with $\eta=10^{-3}$. For $t\geq N$, the robot estimates $\hat A_N$ from $\tilde\theta(t)$ via \eqref{eq: IV},
forecasts $\hat\theta(t)=\hat{\bar\theta}_N+\hat A_N^{\,t-(N-k)}(\tilde\theta(N-k)-\hat{\bar\theta}_N)$, recovers
$\hat H(t),\hat{\tilde b}(t)$ from $\hat\theta(t)$, and computes
$\hat x^\star(t)=\hat H(t)^{-1}\hat{\tilde b}(t)$.
Figure~\ref{fig: rmse_vs_N} shows the predicted and optimal trajectories for
each $N$. 
Figure~\ref{fig: tracking_trajectories}(a) reports RMSE versus $N$ over $[T_{\text{eval}},T]$. We observe that 
mean error decreases monotonically as the training horizon grows.

\vspace{-2mm}

\subsection{Road-Congestion Control (Nonlinear Cost Function)}\label{eq: NCF}

We consider a road-congestion control problem in which
$x=[x_1,x_2]^\T\in\mathbb{R}^2$ denotes the control inputs of two
intersecting corridors. Congestion is modeled through six directional
softplus features $g_i(x)=\log(1+\exp(a_i^\T x-d_i))$, $i\in\{1,\dots,6\}$,
with $d_i=0.5$, directions $a_1,\dots,a_4$ aligned with the coordinate axes,
and $a_5=(1,1)/\sqrt2$, $a_6=(1,-1)/\sqrt2$ capturing the diagonal corridor
interactions. Letting $\theta_i(t)$ be the unknown, time-varying congestion
weight on the $i$-th feature, the cost is
\begin{align}
f(x,\theta(t)) = \tfrac{\theta_0(t)}{2}\|x\|_2^2 + \sum_{i=1}^{6}\theta_i(t)\,g_i(x),
\end{align}
so that $\nabla_x^2f(x,\theta)=\theta_0I_2+\sum_{i=1}^{6}\theta_i\,
\sigma'(a_i^\T x-d_i)\,a_ia_i^\T$ with $\sigma(z)=1/(1+e^{-z})$. Since
$\sigma'>0$ vanishes as $\lVert x\rVert_2\to\infty$, the quadratic weight
sets the curvature floor: $\theta\in\mathcal{K}_\mu$ whenever
$\theta_0\geq\mu$ and $\theta_i\geq0$ for $i\geq1$.
 The parameter vector
$\theta(t) = [\theta_0(t),\ldots,\theta_6(t)]^\T\in\mathbb{R}^7$ evolves under
\eqref{eq: dynamics}.

 We set $\bar\theta=[5,\,0.2,\,0.2,\,0.2,\,0.2,\,0.2,\,0.2]^\T$ and $\mu=1$,
giving a margin $\Delta=4$; with $\sqrt{\trace(\Sigma_\theta)}=0.45$, the
margin-to-noise condition of Remark~\ref{rem: margin} holds over the horizon,
so $\theta(t)\in\Theta$ with high probability. The remaining settings are
$[T_{\mathrm{eval}},T]=[100,200]$, $N=80$, $k=10$, $\sigma_m=0.50$,
$\sigma_p=0.1$, and $M=30$ Monte Carlo trials, other settings follows
Section~\ref{eq: RTTT}. Figure~\ref{fig: congestion_param}(a) shows a
representative trajectory of $\theta_1(t)$, with the other components
behaving similarly, and panel~(b) reports $\lVert\hat A_N-A\rVert_F$
decreasing monotonically in $N$, consistent with Lemma~\ref{lem: IV}.
Figure~\ref{fig: rmse_vs_N}(b) shows the corresponding decrease in future
tracking RMSE.

\begin{figure}[!t]
  \centering
  \includegraphics[width=1\columnwidth,trim={0cm 0cm 0cm 0cm},clip]{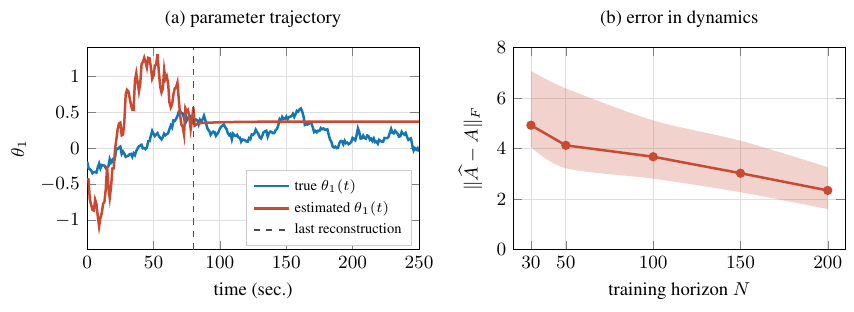}
  \caption{The figure shows the parameter estimation and dynamics identification for the
road-congestion example. (a) Representative trajectory of $\theta_1(t)$:
true (blue solid), estimated from noisy gradients for $t<N$, and forecast
from the identified dynamics for $t\geq N$ (red dashed); the vertical line
marks the end of data collection at $N=80$. (b) Identification error
$\lVert\hat A_N-A\rVert_F$ versus training horizon
$N\in\{30,50,100,150,200\}$, averaged over $M=30$ Monte Carlo trials. This shows more measurements improves latent dynamics identification.}
  \label{fig: congestion_param}
\end{figure}

\section{Conclusion}\label{section5}
We studied time-varying optimization problems whose costs depend on latent parameters evolving under unknown linear stochastic dynamics, with only noisy gradient measurements available. We proposed a three-stage method that reconstructs the latent parameters from windowed gradient measurements, identifies their dynamics via a reduced-bias instrumental-variable estimator that compensates for noisy regressors, and forecasts future parameters to compute the corresponding optimizer. Our finite-sample tracking-error bound separates the contributions of parameter reconstruction, dynamics identification, and future process noise, clarifying how performance depends on the number of gradient samples, the prediction horizon, and the stochasticity of the latent dynamics. Future work includes extensions to constrained optimization, exogenously driven dynamics, and nonlinear parameter dynamics.

\vspace{-2mm}

\bibliographystyle{unsrt}
\bibliography{alias,Main,FP,New}


\end{document}